\documentclass{amsart}
\usepackage{graphicx}
\usepackage{url}
\usepackage{amssymb}
\usepackage[dvipdfmx]{hyperref}
\usepackage{amssymb}
\usepackage{url}
\usepackage{cite}

\newtheorem{theorem}{Theorem}[section]
\theoremstyle{plain}

\newtheorem{lemma}{Lemma}[section]

\newtheorem{problem}{Problem}

\numberwithin{equation}{section}
\begin{document}
\title[ A method to compute the strength using bounds]{ A method to compute the strength using bounds}
\author{ Rikio Ichishima}
\address{Department of Sport and Physical Education, Faculty of Physical
Education, Kokushikan University, 7-3-1 Nagayama, Tama-shi, Tokyo 206-8515,
Japan}
\email{ichishim@kokushikan.ac.jp}
\author{Francesc A. Muntaner-Batle}
\address{Graph Theory and Applications Research Group, School of Electrical
Engineering and Computer Science, Faculty of Engineering and Built
Environment, The University of Newcastle, NSW 2308 Australia }
\email{famb1es@yahoo.es}
\author{Yukio Takahashi}
\address{Department of Science and Engineering, Faculty of Electronics and
Informatics, Kokushikan University, 4-28-1 Setagaya, Setagaya-ku, Tokyo
154-8515, Japan}
\email{takayu@kokushikan.ac.jp}

\begin{abstract}
A numbering $f$ of a graph $G$ of order $n$ is a labeling that assigns distinct elements of the set $\{1,2, \ldots, n \}$ to the vertices of $G$. 
The strength $\mathrm{str}\left(G\right) $ of $G$ is defined by 
$\mathrm{str}\left( G\right) =\min \left\{ \mathrm{str}_{f}\left( G\right)\left\vert f\text{ is a numbering of }G\right. \right\}$,
where $\mathrm{str}_{f}\left( G\right) =\max \left\{ f\left( u\right)
+f\left( v\right) \left\vert uv\in E\left( G\right) \right. \right\} $.
A few lower and upper bounds for the strength are known and, although it is in general hard to compute the exact value for the strength, a reasonable approach to this problem is to study for which graphs a lower bound and an upper bound for the strength coincide. 
In this paper, we study general conditions for graphs that allow us to determine which graphs have the property that lower and upper bounds for the strength coincide and other graphs for which this approach is useless.

\end{abstract}

\date{Nov 23, 2023}
\maketitle

\section{Introduction}

Only graphs without loops or multiple edges will be considered in this
paper. Undefined graph theoretical notation and terminology can be found in \cite{CL} unless otherwise specified. 
The \emph{vertex set} of a graph $G$ is denoted by $V \left(G\right)$, 
while the \emph{edge set} of $G$ is denoted by $E\left (G\right)$.

For the sake of notational convenience, we will use the notation $\left[ a,b\right] $ for the interval of integers $%
x $ such that $a\leq x\leq b$.
A \emph{numbering} $f$ of a graph $G$ of
order $n$ is a labeling that assigns distinct elements of the set $\left[ 1,n%
\right] $ to the vertices of $G$. The \emph{strength} $\mathrm{str}\left(
G\right) $ of $G$ is defined by 
\begin{equation*}
\mathrm{str}\left( G\right) =\min \left\{ \mathrm{str}_{f}\left( G\right)
\left\vert f\text{ is a numbering of }G\right. \right\} \text{,}
\end{equation*}%
where $\mathrm{str}_{f}\left( G\right) =\max \left\{ f\left( u\right)
+f\left( v\right) \left\vert uv\in E\left( G\right) \right. \right\} $. 
Since empty graphs $nK_{1}$ do not have edges, this definition does not apply
to such graphs. 
Consequently, we may define $\mathrm{str}\left( nK_{1}\right) =+\infty $ for every positive integer $n$. 
Therefore, for every nonempty graph $G$ of order $n$, 
it follows that $3 \leq \mathrm{str}\left( G\right) \leq 2n-1$. 
This type of numberings was introduced in \cite{IMO1} as a generalization of the
problem of finding whether a graph is super edge-magic or not (see \cite%
{ELNR} for the definition of a super edge-magic graph, and also consult either 
\cite{AH} or \cite{FIM} for alternative and often more useful definitions of the same
concept). 
A necessary and sufficient condition for a graph to be super edge-magic established by Figueroa-Centeno et al. \cite{FIM} 
gives rise to the concept of the consecutive strength labeling of a graph (see \cite{IMO1} for the definition of a consecutive strength labeling of a graph), 
which is equivalent to the concept of super edge-magic labeling.
For further knowledge on the strength of graphs and related concepts, the authors suggest that
the reader may consult the results in \cite{ GLS, IMO2, IMO3, IMO4, IMOT, IMT, IMT2, IOT1, IOT2, IOT3}.

The \emph{degree of a vertex} $v$ in a graph $G$ is the number of edges of $G$ incident with $v$. The \emph{minimum degree} of $G$ is the minimum degree among the vertices of $G$ and is denoted by $\delta\left(G\right)$.

The next result found in \cite{IMO1} provides a lower bound for the strength of a graph in terms of its order and minimum degree.

\begin{lemma}
\label{lower_bound} For every graph $G$ of order $n$ with $\delta\left(G\right) \geq 1$,
\begin{equation*}
\mathrm{str}\left(G\right) \geq n+\delta\left(G\right) \text{.}
\end{equation*}
\end{lemma}

The lower bound given in Lemma \ref{lower_bound} is sharp in the sense that there are infinitely many graphs $%
G$ for which $\mathrm{str}\left( G\right) =\left| V\left( G\right) \right|+\delta \left( G\right) $ (see \cite{GLS, IMO1, IMO2, IMOT, IOT2, IOT3, IMT, IMT2} for a detailed list of such graphs and other sharp bounds). 
In fact, it was shown in \cite{IMO4} that for every $k\in \left[ 1,n-1\right] $, 
there exists a graph $G$ of order $n$ satisfying $\delta \left( G\right) =k$ and 
$\mathrm{str}\left( G\right) =n+k$.

A set $S$ of vertices in a graph $G$ is \emph{independent} if no two vertices in $S$ are adjacent. 
The maximum number of vertices in an independent set of vertices of $G$ is called the \emph{independence number} of $G$ and is denoted by $\beta\left(G\right)$. 
From the definition, it is immediate that if $G$ and $H$ are two graphs of the same order such that $G \subseteq H$, then $\beta\left(H\right) \leq \beta\left(G\right)$.

The next result provides lower and upper bounds for the strength of a graph in terms of its order and independence number. 
The lower bound is due to Gao et al. \cite{GLS} and the upper bound was found in \cite{IMT2}.

\begin{lemma}
\label{upper_bound} For every graph $G$ of order $n$,
\begin{equation*}
2n-2\beta\left(G\right)+1 \leq \mathrm{str}\left(G\right) \leq 2n-\beta\left(G\right) \text{.}
\end{equation*}
\end{lemma} 

Since there is only one numbering $f$ of the complete graph $K_{n}$ ($n \geq 2$) and the label of the edge joining the vertices labeled $n$ and $n-1$ is $2n-1$, it follows that $\mathrm{str}_{f}\left( K_{n}\right)= \mathrm{str}\left(K_{n}\right)= 2n-1$. 
It is also clear that $\beta\left(K_{n}\right)=1$ for all positive integers $n$. 
Therefore, the lower and upper bounds given in Lemma \ref{upper_bound} coincide when $G=K_{n}$ for integers $n \geq 2$ and this is the only case when these two bounds coincide.

There are other related parameters that have been studied in the area of
graph labelings. Excellent sources for more information on this topic are
found in the extensive survey by Gallian \cite{Gallian}, which also includes
information on other kinds of graph labeling problems as well as their
applications.

\section{Computing the strength of graphs}
This section is devoted to study for which graphs the lower and upper bounds provided in Lemmas \ref{lower_bound} and \ref{upper_bound} coincide.
The next result shows that the lower bound given in Lemma \ref{lower_bound} coincides with the upper bound given in Lemma \ref{upper_bound} for complete $k$-partite graphs.

\begin{lemma}
\label{result_1}
Let $G$ be a complete $k$-partite graph of order $n$ for some positive integer $k$. 
Then 
\begin{equation*}
n+\delta\left(G\right)=2n-\beta\left(G\right) \text{.}
\end{equation*}
\end{lemma}
\begin{proof}
Let the partite sets of $G$ be $V_{1},V_{2}, \dots, V_{k}$ for some positive integer $k$. 
Then 
\begin{equation*}
\beta\left(G\right)=\max \left\{\left\vert V_{i}\right\vert\left\vert i \in \left[1,k\right] \right. \right\}
\end{equation*}
and $n=\sum_{i=1}^{k} \left\vert V_{i}\right\vert$.
Therefore, 
\begin{equation*}
\delta\left(G\right)=\sum_{n=1}^{k} \left\vert V_{i}\right\vert-\max \left\{\left\vert V_{i}\right\vert\left\vert i \in \left[1,k\right] \right. \right\}
=n-\beta\left(G\right) \text{,}
\end{equation*} 
implying that 
\begin{equation*}
n+\delta\left(G\right)=n+\left(n-\beta\left(G\right)\right)=2n-\beta\left(G\right) \text{.}
\end{equation*}
\end{proof}

The next result concerns with the partition of the vertex set of a connected bipartite graph.

\begin{lemma}
\label{result_2}
Let $G$ be a connected bipartite graph with partite sets $U$ and $W$. Then the partition $U \cup W$ is unique up to renaming of vertices.
\end{lemma}

A set of edges in a graph $G$ is \emph{independent} if no two edges in the set are adjacent in $G$. 
The edges in an independent set of edges of $G$ are called a \emph{matching} in $G$. 
If $M$ is a matching in a graph $G$ with the property that every vertex of $G$ is incident with an edge of $M$, then $M$ is a \emph{perfect matching} in $G$.

In the following, we determine the independence number of a connected bipartite graph that contains a matching.

\begin{lemma}
\label{result_4}
Let $G$ be a connected bipartite graph having partite sets $U$ and $W$, where $\left\vert U\right\vert=m$ and $ \left\vert W\right\vert=n$ for some positive integers $m$ and $n$ with $m \geq n$. 
If $G$ contains a matching of order $2n$, then $\beta\left(G\right)=m$.
\end{lemma}
\begin{proof}
Since $G$ is a connected bipartite graph, it follows from Lemma \ref{result_2} that the partition of $V\left(G\right)$ is unique.
With this fact, we consider two cases.

\noindent \textbf{Case 1.} For $m=n$, by hypothesis, $G$ contains a perfect matching $M$ of order $2m$.
It is clear that $\beta\left(M\right)=m$. 
Since $\left\vert V\left(M\right)\right\vert= \left\vert V\left(G\right)\right\vert=2m$ and $M \subseteq G$, 
it follows that $\beta\left(G\right) \leq \beta\left(M\right)=m$. 
The sets $U$ and $W$ are both independent sets of cardinality $m$, implying that
$\beta\left(G\right) \geq m$.
Therefore, we conclude that $\beta\left(G\right)=m$.

\noindent \textbf{Case 2.} For $m>n$, by hypothesis, $G$ contains a matching $M$ of order $2n$ in $G$.
Since $M$ has order $2n$, it follows that $n$ vertices of $U$ and $n$ vertices of $W$ are vertices of the matching.
Then the independent set may contain only $n$ vertices belonging to the subgraph of $G$ induced by the vertices of the matching $M$ together with the remaining vertices of $W$ not in $M$.
Thus, any independent set of $G$ may contain at most $n+\left(m-n\right)$ vertices.
On the other hand, $U$ is an independent set of cardinality $m$ in $G$.
Therefore, we conclude that $\beta\left(G\right)=m$.
\end{proof}

The next result provides a sufficient condition for a connected bipartite graph $G$ that guarantees that 
$\left\vert V\left(G\right)\right\vert +\delta\left(G\right) \neq 2\left\vert V\left(G\right)\right\vert -\beta\left(G\right)$.

\begin{theorem}
\label{theorem_Francesc}
Let $G$ be a connected bipartite graph having partite sets $U$ and $W$, where $\left\vert U\right\vert=m$ and $ \left\vert W\right\vert=n$ for some positive integers $m$ and $n$ with $m \geq n$. 
If $G$ is not the complete bipartite graph $K_{m.n}$ and contains a perfect matching of order $2n$, 
then 
\begin{equation*}
\left\vert V\left(G\right)\right\vert +\delta\left(G\right) \neq 2\left\vert V\left(G\right)\right\vert -\beta\left(G\right) \text{.}
\end{equation*}
\end{theorem}
\begin{proof}
Assume the hypothesis of the theorem.
It then follows from Lemma \ref{result_4} that 
\begin{equation*}
\beta\left(G\right)=m=\beta\left(K_{m,n}\right) \text{.}
\end{equation*}
Since $G$ is not the complete bipartite graph $K_{m,n}$, it follows that 
\begin{equation*}
\delta\left(G\right)<n=\delta\left(K_{m,n}\right) \text{,}
\end{equation*}
implying by Lemma \ref{result_1} that 
\begin{equation*}
\left\vert V\left(G\right)\right\vert+\delta\left(G\right)<\left\vert V\left(G\right)\right\vert+n=2\left\vert V\left(G\right)\right\vert-\beta\left(G\right) \text{.}
\end{equation*}
Therefore, the result follows.
\end{proof}

The preceding result is not necessarily true for $k$-partite graphs when $k \geq 3$. 
To see this, let $G$ be a complete $k$-partite graph $K_{n_{1},n_{2}, \dots, n_{k}}$ with partite sets $V_{1},V_{2}, \dots, V_{k}$ such that $\left\vert V_{i}\right\vert=n_{i}$ for each $i \in \left[1,k\right]$, where $n_{1}=\max\left\{n_{i} \left\vert \right. i\in \left[1,k\right]\right\}$ and $n_{1} \geq \sum_{i=2}^{k} n_{i}$.
Consider the subgraph $H$ of $G$ induced by the set of vertices $\bigcup_{i=1}^{k} V_{i}$, and let $H^{\prime}$ be any subgraph of $H$.
If we consider any subgraph $G^{\prime}$ of $G$ with 
\begin{equation*}
V\left(G^{\prime}\right)=V\left(G\right) \text{ and } E\left(G^{\prime}\right)=E\left(G\right)-E\left(H^{\prime}\right) \text{,}
\end{equation*}
then we have the following result.

\begin{theorem}
\label{theorem_Francesc2}
Let $G^{\prime}$ be any subgraph of the complete $k$-partite graph $G$ as defined above.
Then $\mathrm{str}\left( G^{\prime}\right)=\mathrm{str}\left( G\right)$.
\end{theorem}
\begin{proof}
First, observe that  
$\deg_{G} v=\sum_{i=1}^{k} n_{i}-n_{j}$ for all $v \in V_{j}$, where $j \in \left[1,k\right]$.
Further, observe that $\left\vert V\left(G\right)\right\vert=\sum_{i=1}^{k} n_{i}$, $\beta\left(G\right)=n_{1}$ and $\delta\left(G\right)=\sum_{i=2}^{k} n_{i}$.
This implies by Lemmas \ref{lower_bound} and \ref{upper_bound} that 
\begin{equation*}
\left\vert V\left(G\right)\right\vert+\delta\left(G\right) \leq \mathrm{str}\left( G\right) \leq 2\left\vert V\left(G\right)\right\vert-\beta\left(G\right)
\end{equation*}
or, equivalently, 
\begin{equation*}
2\sum_{i=1}^{k} n_{i}-n_{1}=\sum_{i=1}^{k} n_{i}+\sum_{i=2}^{k} n_{i} \leq \mathrm{str}\left( G\right) \leq 2\sum_{i=1}^{k} n_{i}-n_{1} \text{.}
\end{equation*}
Consequently, $\mathrm{str}\left( G\right) = 2\sum_{i=1}^{k} n_{i}-n_{1}$.

Now, consider the graph $G^{\prime}$ as defined above.
Since no edges $xy$, where $x \in V_{1}$ have been removed and any $y \in V\left(H\right)-V_{1}$ is adjacent in $G^{\prime}$ to all vertices of $V_{1}$, it follows that all these vertices have degree at least $n_{1}$ and $n_{1} \geq \sum_{i=2}^{k} n_{i}$, implying that
$\delta\left(G^{\prime}\right)=\sum_{i=2}^{k} n_{i}$.
It is clear that $\beta\left(G^{\prime}\right)=n_{1}$ since $n_{1} \geq \sum_{i=2}^{k} n_{i}$.
It is also true that $\left\vert V\left(G^{\prime}\right)\right\vert=\left\vert V\left( G\right)\right\vert$.
Therefore, we conclude that $\mathrm{str}\left( G^{\prime}\right) = 2\sum_{i=1}^{k} n_{i}-n_{1}$.
\end{proof}

\section{Conclusions}
The goal of this paper is to study graphs for which the lower bounds and the upper bounds for the strength coincide, and hence the strength of such graphs can be computed.
We have seen this in Theorem \ref{theorem_Francesc2}.
We also find graphs for which we show that this approach does not work (see Theorem \ref{theorem_Francesc}).
Notice that Theorem \ref{theorem_Francesc} suggests the use of a well-known theorem due to Hall \cite {Hall} and referred to as Hall's Marriage Theorem to determine the existence of certain matchings in bipartite graphs.
Finally, we conclude this paper with the following two open problems that we think that may be interesting for future research.

Let $KG_{n,k}$ denote the Kneser graph introduced by Lov\'{a}sz \cite{Lovasz} to prove Kneser's Conjecture (see \cite{Kneser}).
For positive integers $k$ and $n$ with $n=k+1$, $KG_{n,k}=nK_{1}$, it follows from the definition that $\mathrm{str}\left(KG_{n,k}\right)=+\infty$.
For positive integers $k$ and $n$ with $n \geq 2k$, we propose the following problem.

\begin{problem}
For every two positive integers $k$ and $n$ with $n \geq 2k$, determine the exact value of $\mathrm{str}\left( KG_{n,k}\right)$.
\end{problem}

Let $\overline{KG_{n,k}}$ denote the complement of $KG_{n,k}$. 
For $k=1$ and any positive integer $n$, $\overline{KG_{n,k}}=\overline{K_{n}}=nK_{1}$, and follows from the definition that $\mathrm{str}\left(\overline{KG_{n,k}}\right)=+\infty$.
For integers $k$ and $n$ with $n \geq 2k \geq 4$, we propose the following problem.

\begin{problem}
For every two integers $k$ and $n$ with $n \geq 2k \geq 4$, determine the exact value of $\mathrm{str}\left( \overline{KG_{n,k}}\right)$.
\end{problem}

We find the above two problems interesting since Kneser graphs and their complements are very well studied graphs, 
and hence any question related to them seems to be interesting.


\begin{thebibliography}{99}
\bibitem{AH} B.D. Acharya and S.M. Hegde, Strongly indexable graphs, 
\textit{Discrete Math}., \textbf{93} (1991) 123--129.

\bibitem{CL} G. Chartrand and L. Lesniak, \textit{Graphs \& Digraphs} 3th ed. CRC Press, 1996.

\bibitem{ELNR} H. Enomoto, A. Llad\'{o}, T. Nakamigawa, and G. Ringel, Super
edge-magic graphs, \textit{SUT J. Math.}, \textbf{34} (1998) 105--109.

\bibitem{FIM} R.M. Figueroa-Centeno, R. Ichishima, and F.A. Muntaner-Batle,
The place of super edge-magic labelings among other classes of labelings, 
\textit{Discrete Math}., \textbf{231} (2001) 153--168.

\bibitem{Gallian} J.A. Gallian, A dynamic survey of graph labeling, \textit{%
Electron. J. Combin}., (2022) \#DS6.

\bibitem{GLS} Z.B. Gao, G.C. Lau, and W.C. Shiu,
Graphs with minimal strength, \textit{Symmetry}, \textbf{13} (2021) 513. 

\bibitem{Hall} P. Hall, On representation of subsets, \textit{London Math. Soc}., 
\textbf{10} (1935) 26--30.

\bibitem{IMO1} R. Ichishima, F.A. Muntaner-Batle, and A. Oshima, Bounds for
the strength of graphs, \textit{Australas. J. Combin}., 
\textbf{72} (3) (2018) 492--508.

\bibitem{IMO2} R. Ichishima, F.A. Muntaner-Batle, and A. Oshima, The
strength of some trees, \textit{AKCE Int. J. Graphs Comb}., 
\textbf{17} (1) (2020) 486--494. 

\bibitem{IMO3} R. Ichishima, F.A. Muntaner-Batle, and A. Oshima, Minimum
degree conditions for the strength and bandwidth of graphs,
\textit{Discrete Appl. Math}., \textbf{340} (2022) 191--198.

\bibitem{IMO4} R. Ichishima, F.A. Muntaner-Batle, and A. Oshima, 
A result on the strength of graphs by factorizations of complete graphs, 
\textit{Discrete Math. Lett}., \textbf{8} (2022) 78--82.

\bibitem{IMOT} R. Ichishima, F.A. Muntaner-Batle, A. Oshima, and Y. Takahashi, The
strength of graphs and related invariants, \textit{Memoirs Kokushikan Univ. Inf. Sci}., 
\textbf{41} (2020) 1--8.

\bibitem{IMT} R. Ichishima, F.A. Muntaner-Batle, and Y. Takahashi, 
On the strength and independence number of graphs, \textit{Contrib. Math}., 
\textbf{6} (2022) 25--29.

\bibitem{IMT2} R. Ichishima, F.A. Muntaner-Batle, and Y. Takahashi, 
 On the strength of powers of paths and cycles, preprint.

\bibitem{IOT1} R. Ichishima, A. Oshima, and Y. Takahashi, Bounds for the
edge-strength of graphs, \textit{Memoirs Kokushikan Univ. Inf. Sci}., 
\textbf{41} (2020) 9--15.

\bibitem{IOT2} R. Ichishima, A. Oshima, and Y. Takahashi, The edge-strength
of graphs, \textit{Discrete Math. Lett}., \textbf{3} (2020) 44--49.

\bibitem{IOT3} R. Ichishima, A. Oshima, and Y. Takahashi, 
Some new results on the edge-strength and strength of graphs, 
\textit{Discrete Math. Lett}., \textbf{12} (2023) 22--25.

\bibitem{Kneser} M. Kneser, Aufgabe 300, \textit{Jahresber. Dtsch. Math. Ver}., \textbf{58} (1955) 27.

\bibitem{Lovasz} L. Lov\'{a}sz, Kneser's conjecture, chromatic numbers and
homotop, \textit{J. Combin. Theory, Ser. A}, \textbf{25} (1978) 319--324.

\end{thebibliography}
\end{document}